\documentclass{amsart}
\usepackage{packcomplet}

\newcommand{\acm}{\overline{ac}_m}
\newcommand{\OK}{\mathcal{O}_K}
\newcommand{\MK}{\mathcal{M}_K}
\newcommand{\ord}{\operatorname{ord}}

\def\QQ{{\mathbb Q}}
\def\RR{{\mathbb R}}
\def\ZZ{{\mathbb Z}}
\def\ac{{\overline{ac}}}
\newcommand{\sq}{\mathrel{\square}}

\author{Raf Cluckers}
\address{Universit\'e Lille 1, Laboratoire Painlev\'e, CNRS - UMR 8524, Cit\'e Scientifique, 59655
Villeneuve d'Ascq Cedex, France, and,
Katholieke Universiteit Leuven, Department of Mathematics,
Celestijnenlaan 200B, B-3001 Leu\-ven, Bel\-gium}
\email{Raf.Cluckers@math.univ-lille1.fr}
\urladdr{http://math.univ-lille1.fr/$\sim$cluckers}

\author{Florent Martin}
\address{Fakult\"{a}t f\"{u}r Mathematik, Universit\"{a}t Regensburg, 93040 Regensburg, Germany}
\email{florent.martin@mathematik.uni-regensburg.de}
\urladdr{http://homepages.uni-regensburg.de/$\sim$maf55605/}

\date{}
\begin{document}
\title[Definable, $p$-adic extensions of Lipschitz maps]{
A definable, $p$-adic  analogue of Kirszbraun's Theorem on extensions of Lipschitz maps}

\begin{abstract}
A direct application of Zorn's Lemma gives that every Lipschitz map $f:X\subset \mathbb{Q}_p^n\to \mathbb{Q}_p^\ell$ has an extension
 to a Lipschitz map $\widetilde f: \mathbb{Q}_p^n\to \mathbb{Q}_p^\ell$. This is analogous, but more easy, to Kirszbraun's Theorem about the existence of Lipschitz extensions of Lipschitz maps $S\subset \mathbb{R}^n\to \mathbb{R}^\ell$. Recently, Fischer and Aschenbrenner obtained a definable version  of Kirszbraun's Theorem. In the present paper, we prove in the $p$-adic context that $\widetilde f$ can be taken definable when $f$ is definable, where definable means semi-algebraic or subanalytic (or, some intermediary notion). We proceed by proving the existence of definable, Lipschitz retractions of $\mathbb{Q}_p^n$ to the topological closure of $X$ when $X$ is definable.
 \end{abstract}


\maketitle


\section{Introduction}

A Lipschitz continuous map $f$ with Lipschitz constant $1$ (a Lipschitz map for short) from any subset $X\subset \mathbb{Q}_p^n$ to $\mathbb{Q}_p^\ell$ can be extended to a Lipschitz map $\widetilde f : \QQ_p^n \to \QQ_p^\ell$, by Zorn's lemma. This is explained in the proof of Theorem 1.2 of \cite{Bask}, the key point 
being that if $X$ is moreover closed and $a\in\QQ_p^n$ is arbitrary, then $f$ can be extended to a Lipschitz map $X\cup \{a\} \to \QQ_p^\ell$ by defining the value of $a$ as $f(x_a)$ for a chosen $x_a\in X$ which lies closest to $a$ among the elements of $X$. By Zorn's lemma and an easy passing to the topological closure like in Lemma \ref{lemma:extension} below, $f$ can thus be extended to a Lipschitz map $\widetilde f: \mathbb{Q}_p^n\to\QQ_p^\ell$.
The aim of this paper is to render this construction of $\widetilde f$ constructive, when more is known about $f$. Such a question was raised to us by M.~Aschenbrenner, after work by him and Fischer \cite{AschF} on making such results constructive (more precisely definable) in the real case.

\par
Let us first briefly recall the real situation, where we refer to \cite{AschF} for a more complete context and history.
A Lipschitz map $g:S\subset \RR^n\to \RR^\ell$ can always be extended to a Lipschitz map $\widetilde g: \RR^n\to\RR^\ell$, but the argument is more subtle than just applying Zorn's lemma. In the case that $\ell=1$, the result was observed by McShane \cite{McShane} and independently by Whitney \cite{Whitneyext} in 1934 and can be explained in terms of moduli of continuity of $f$ (see Proposition 5.4 of \cite{AschF}). The case of general $\ell$ is more delicate and was obtained by Kirszbraun \cite{Kirs} also in 1934, partially relying on Zorn's lemma. Recently, Fischer and Aschenbrenner \cite{AschF} showed that $\widetilde g$ can be taken definable when $g$ is definable (in a very general sense). This can be seen as a constructiveness result. Results related to Whitney's extension theorem continue to play a role in differential topology (see e.g.~\cite{MarDomDiff}).

\par
We prove the definability of $\widetilde f$ in the $p$-adic case when $f$ is definable, where definable can mean semi-algebraic or subanalytic, or some intermediary notion, coming respectively from \cite{Mac}, \cite{DenVDD} and \cite{CluLip}.
We approach our $p$-adic result via showing that, for any closed definable subset $X\subset \QQ_p^n$, there exists a definable Lipschitz retraction $$
r:\QQ_p^n\to X,
$$
namely, a Lipschitz map $r:\QQ_p^n\to X$ such that $r(x)=x$ whenever $x\in X$, see Theorem \ref{theo:retract} below. The existence of Lipschitz retractions in the real case exists onto convex closed sets (see Corollary 2.14 of \cite{AschF}), but not for general closed sets. The general existence of Lipschitz retractions in our setting may be somewhat surprising, but in fact, the absence of a convexity condition in whatever form in the $p$-adic case reminds one of a similar absence in the results on piecewise Lipschitz continuity of \cite{ClHa}. 

\par
In the $p$-adic case, there is in fact no difference in difficulty between the $\ell=1$ case and the case of general $\ell$, by the usual definition of the ultra-metric norm as the sup-norm. Naturally, the case where $n=1$, namely, when the domain of $f$ is a subset of $\QQ_p$, is more easy than the case of general $n$ and has been treated recently in \cite{KuiLip}. We prove our results by an induction on $n$, where we use a certain form of cell decomposition/preparation with Lipschitz continuous centers, similar to such a result of \cite{ClHa} but which treated no form of preparation, see Theorem \ref{theo:celldec} below. This decomposition/preparation result is used to geometrically simplify the set $X$ by replacing it by what we call a centred cell. Once we have reduced to the case that $X$ is a centred cell, we use an almost explicit construction of the Lipschitz retraction $r$, with as only non-explicit part some choices of definable Skolem functions. On the way, we obtain a result on the existence of definable isometries with properties adapted to the geometry of $X$, see Proposition \ref{lemma:monomial}.

\par
Our results also hold in families of definable functions, see the variants given by Theorems \ref{theo:1:fam} and \ref{theo:retract:fam} at the end of the paper, and for any fixed finite field extension $K$ of $\QQ_p$.

\subsection{Main results}


To state our main results we first fix some notation.
We consider a finite extension $K$ of $\Qp$.
We denote by $\ord: K \to \Z \cup \{+\infty \}$ the associated valuation and
$|\cdot| : K \to \R_+$  the associated norm, defined by $|x|= q^{-\ord(x)}$ with $q$ the number of elements of the residue field of $K$.
We equip $K^n$ with the product metric, namely
$\displaystyle d(x,y) = \max_{i=1\ldots n} |x_i-y_i|$ for $x=(x_1 \ldots x_n)$ and $y=(y_1 \ldots y_n)$ in $K^n$, and with the metric topology.

\par
Write $\mathcal{O}_K$ for the valuation ring,
$\mathcal{M}_K$ for the maximal ideal of $K$ and $k_K$ for the residue field.
Let us fix $\varpi$ some uniformizer of $\mathcal{O}_K$.
We denote by
$\overline{ac}_m : K \to \mathcal{O}_K / (\mathcal{M}_K^m)$ the map sending nonzero $x\in K$ to
$x \varpi^{- \ord(x)}$ mod $\mathcal{M}_K^m$, and sending zero to zero.
This map is called the $m$-th angular component map.

\par
We denote by $RV$ the union of $K^\times / (1+\mathcal{M}_K)$ and $\{0\}$ and
by $rv : K\to RV$ the quotient map sending $0$ to $0$.
More generally, if $m$ is a positive integer, we set
$RV_m = K^\times / (1+\mathcal{M}_K^m) \cup \{0\}$ and
$rv_m : K \to RV_m$ the quotient map. \par
For $m,n >0 \in \N$, we set
\[ Q_{m,n} =\{ x\in K^\times \st \ord(x) \in n \Z \ \and \ \acm(x)=1 \}.\]

\par

When $X \subset Y \subset K^n$, a retraction from $Y$ to a  $X$ is a map $r: Y \to X$ which is the identity on $X$.
By definable we mean either semi-algebraic, or, subanalytic, or an intermediary structure given by an analytic structure on $\QQ_p$ as in \cite{CluLip}. The notions of semi-algebraic sets and of subanalytic sets are recalled in \cite{ClHa} and based on quantifier elimination results from \cite{Mac}, \cite{DenVDD}, and we refer to \cite{CluLip} for background on more general analytic structures.
A function
$$
f:X\subset \QQ_p^n\to \QQ_p^\ell
$$
is called Lipschitz (in full, Lipschitz continuous with Lipschitz constant $1$) when
$$
|f(x)-f(y)|\leq |x-y| \mbox{ for all $x,y\in X$}.
$$
We call a function $\widetilde f:A\to B$ an extension of a function $f:A_0\to B$ when $A_0\subset A$ and $\widetilde f$ coincides with $f$ on $A_0$.

\bigskip

The following results, and their family versions given below as Theorems \ref{theo:1:fam} and \ref{theo:retract:fam}, are the main results of the paper.
\begin{theo}[$Ext_n$]
\label{theo:1}
Let $X \subset K^n$ be a definable set.
Let $f:X \to K^\ell$ be a definable and Lipschitz function.
Then there exists a definable Lipschitz $\widetilde{f} :K^n \to K^\ell$ which is an extension of $f$. Moreover, we can ensure that the range of $\widetilde{f}$ is contained in the topological closure of the range of $f$.
\end{theo}

\begin{theo}[$Ret_n$]
\label{theo:retract}
Let $X \subset K^n$ be a definable set.
There exists a definable Lipschitz retraction $r : K^n \to \overline{X}$. Here, $\overline{X}$ is the topological closure of $X$ in $K^n$.
\end{theo}

Theorem \ref{theo:1} can be seen as a consequence of Theorem \ref{theo:retract}, as follows.

\begin{proof}[Proof of Theorem \ref{theo:1} knowing Theorem \ref{theo:retract}.]
\label{proof:theo1}
Let $\overline{f} : \overline{X} \to K$ be the unique definable Lipschitz extension of $f$ to the topological closure $\overline{X}$ of $X$, as given by Lemma \ref{lemma:extension} below.
Let $r : K^n \to \overline{X}$ be a definable Lipschitz retraction as given by Theorem \ref{theo:retract}.
Then $\widetilde{f} = \overline{f} \circ r$ extends $f$, and is Lipschitz.
\end{proof}

In fact, we will prove these two theorems together with Proposition \ref{lemma:monomial} by a joint induction on $n$ in Section \ref{sec:proof}.


\begin{rem}[Some remarks about Theorem \ref{theo:retract}]
$ $
\begin{enumerate}
\item
One really need to consider $\overline{X}$ in Theorem \ref{theo:retract}.
For instance, there is no continuous retraction from $K$ to $K^\times$ .
\item The Archimedean analogue of Theorem \ref{theo:retract} is false.
For instance, there is no continuous retraction $r : \R \to \{-1,1\}$.
However, when $X \subset \R^n$ is a closed convex set, the projection
$r : \R^n  \to X$ to the closest point of $X$ is a Lipschitz retraction, see Corollary 2.14 of \cite{AschF}.
\item
It would be interesting to know if Theorems \ref{theo:1} and \ref{theo:retract} hold in some form for other classes of valued fields $K$.
Natural examples would be $\R((t))$, $\C((t))$ or algebraically closed valued fields (see below).
Some difficulties in more general settings are: the absence of definable Skolem functions in general (they are used in the proof of Theorem \ref{theo:retract}), and, 
infiniteness of the residue field (we use the finiteness of the residue field in Corollary \ref{cor:RVtoZ}).
\item
In this form, the analogue of Theorem \ref{theo:retract} does not hold for ACVF, the theory of algebraically closed valued fields.
Indeed, let $L$ be an algebraically closed valued field,
and let $X = \{x\in L \st |x|>1 \}$.
Then $X$ is a closed set,
but one can check that there is no Lipschitz retraction $r : L \to X$.
However, in this example $X$ might not be considered as a closed set, because it is defined by means of $<$.
One might hope that for a "good" notion of definable closed set (such as a set defined with $\leq$, $=$, finite unions and intersections), an analogue of Theorem \ref{theo:retract}
holds in ACVF.
For instance, there exists a definable Lipschitz retraction from $L$ onto $\{x\in L \st |x|\geq 1 \}$.
\end{enumerate}
\end{rem}

\subsection{Acknowledgement}
We would like to thank M.~Aschenbrenner for raising the question of $p$-adic definable Lipschitz extensions to us, T.~Kuijpers for enthusiastic discussions about this problem in the $n=1$ case
and P. Cubides for stimulating discussions. The authors were supported by the European Research Council under the European Community's Seventh Framework Programme (FP7/2007-2013)
with ERC Grant Agreements nr. 615722 MOTMELSUM, by SFB 1085 "Higher invariants" and by the Labex CEMPI (ANR-11-LABX-0007-01). We would also like to thank 
the FIM of the ETH in Z\"urich, where part of the research was done.


\section{Preliminary results}

\subsection{Dimension of definable sets}
To a nonempty definable set $X \subset K^n$, one can associate a dimension, denoted by $\dim(X) \in \N$. It is defined as the maximum of the numbers $k\geq 0$ such that there is a coordinate projection $p:K^n\to K^k$ such that $p(X)$ has nonempty interior in $K^k$.
This dimension, studied in \cite{HaMaAver} in the slightly more general context of P-minimal structures, enjoys some nice properties that we will freely use.
Let us mention that $\dim(\overline{X})=\dim(X)$ and that if $f : X \to Y$ is a definable map then $\dim(f(X)) \leq \dim(X)$.

\subsection{Presburger sets}
\label{section:Presburger}
A Presburger set is a subset of $\Z^n$ defined in the language $\mathcal{L}_{Pres}$ consisting of $+,-,0,1,<$ and, for each $n>0$,  the binary relation $\cdot\equiv_n\cdot$ for congruence modulo $n$.
Since $(\Z, \mathcal{L}_{Pres})$ eliminates quantifiers by \cite{Presburger}, any Presburger set can be seen to be a finite union of sets of the following form:
\[ \{ (u,v)\in \Z^{n} \st u\in U,\ v\in a\Z+b  \ \and \ \alpha(u) \sq_1 v \sq_2 \beta(u) \}\]
where $U \subset \Z^{n-1}$ is a Presburger set, $a,b$ are positive integers, $\alpha, \beta : U \to \Z$ are definable functions and each $\sq_i$ is $<$ or no condition; see \cite{CPres} for related results and some background.

\subsection{Retractions}

We begin with three basic lemmas.

\begin{lemma}
\label{lemma:extension}
Let $X \subset K^n$ and $Y\subset K^\ell$ be definable sets.
Let $f : X\ \to Y$ be a definable Lipschitz function.
There exists a unique Lipschitz extension
$\overline{f} : \overline{X} \to \overline{Y}$ which is also definable.
\end{lemma}
\begin{proof}
Simple topological and definability argument.
\end{proof}

\begin{lemma}
\label{lemma:retract}
Let $X\subset Y\subset K^n$.
Let $r : Y \to X$ be a Lipschitz retraction.
Then for all $y\in Y$, $|r(y) -y| = d(y,X)$, where $d(y,X)$ is the infimum of $d(y,x)$ for $x\in X$.
\end{lemma}

\begin{proof}
Let us assume that $|r(y) -y|>d(y,X)$, and let $x\in X$ such that
$|r(y)-y|>|x-y|$.
Then
\[|r(y)-r(x)| = |r(y) -x| =|(r(y) -y) +(y-x)| = |r(y)-y|>|y-x|.\]
This contradicts the fact that $r$ is Lipschitz.
\end{proof}

The following result is inspired by \cite{KuiLip}[Lemma 11].

\begin{lemma}[Gluing Lemma]
\label{lemma:glue}
Let $X \subset K^n$ be a definable set.
Let $X_i \subset X$ for $i=1, \ldots, m$ be a finite collection of definable sets and let
$r_i : X \to X_i$ be definable Lipschitz retractions.
Then there exists a definable Lipschitz retraction
\[r: X\to\bigcup\limits_{i=1}^mX_i.\]
\end{lemma}

\begin{proof}
With an easy induction on $m$, we can assume that $m=2$.
So, we have two definable sets $X_1, X_2$.
Let us define $r$ by
\[
\begin{array}{cccc}
r:  & X & \to & X_1 \cup X_2 \\
& x & \mapsto &
\begin{cases}  r_1(x) & \text{if}  \ d(x,X_1) \leq d(x,X_2) \\
r_2(x) & \text{otherwise}.
\end{cases}
\end{array}
\]
Let $x,y \in X$ and let us prove that $|r(x)-r(y)| \leq |x-y|$.\\

\textbf{Case 1}.  $d(x,X_1) \leq d(x,X_2)$ and $d(y,X_1)\leq d(y,X_2)$.\\

Then $r(x)= r_1(x)$ and $r(y) = r_1(y)$, so $|r(x)-r(y)| \leq |x-y|$ because $r_1$ is Lipschitz. \\

\textbf{Case 2}.  $d(x,X_1) > d(x,X_2)$ and $d(y,X_1) > d(y,X_2)$.\\

Then $r(x)= r_2(x)$ and $r(y) = r_2(y)$, so $|r(x)-r(y)| \leq |x-y|$ because $r_2$ is Lipschitz. \\

\textbf{Case 3}. $d(x,X_1) \leq d(x,X_2)$ and $d(y,X_1)> d(y,X_2)$.\\

 This implies that $r(x)=r_1(x)\in X_1$ and $r(y)=r_2(y)\in X_2$.
 We obtain:
\begin{equation}
\label{eq1}
|x-r(y)| \geq d(x,X_2) \geq d(x,X_1) =|r(x)-x|.
\end{equation}
The last equality follows from Lemma \ref{lemma:retract}.
Then
\begin{equation}
\label{eq2}
|x-r(y)| \geq \max(|r(x)-x|,|x-r(y)|) \geq |r(x)-r(y)|
\end{equation}
where the first inequality follows from \eqref{eq1}.
Moreover
\begin{equation}
\label{eq3}
|y-r(x)| \geq d(y,X_1) > d(y,X_2) = |y-r(y)|.
\end{equation}
The last equality follows from Lemma \ref{lemma:retract} again.
Then
\begin{equation}
\label{eq4}
|r(y)-r(x)| =|(r(y)-y)+(y-r(x))| \overset{\eqref{eq3}}{=} |y-r(x)|.
\end{equation}
With one more step, we get
\begin{equation}
\label{eq5}
|x-r(y)| \overset{\eqref{eq2}}{\geq} |r(x)-r(y)| \overset{\eqref{eq4}}{=} |y-r(x)| \overset{\eqref{eq3}}{>} |y-r(y)|.
\end{equation}
Finally
\begin{equation*}
|x-y| = |(x-r(y))+(r(y)-y)| \overset{\eqref{eq5}}{=}  |x-r(y)| \overset{\eqref{eq2}}{\geq} |r(x)-r(y)|,
\end{equation*}
which is what we have to show.
\end{proof}

\subsection{Centred cells}

\begin{defi}[Centred cells]
\label{defi:centredcell}
Let $m$ and $n'\leq n$  be integers.
We say that a definable set $C\subset (K^\times)^n$ is an open centred cell if it is of the form
\[C = rv_m^{-1}(G) \cap (K^\times)^n,\]
for some set $G \subset (RV_m )^n$, that is,
\[
C = \{x\in (K^\times)^n \mid (rv_m(x_1),\ldots,rv_m(x_n))\in G  \}.
\]
Furthermore, if $C' \subset K^{n'}$ is an open centred cell,  we say that
\[C := C' \times \{\underbrace{(0,\ldots,0)}_{n-n' \ \text{times}}\} \subset K^{n}\]
is a centred cell.
\end{defi}

A centred cell is thus just a pullback under $rv_m$ of some definable subset $G$ of $RV_m$ for some $m$,
where we call a subset $A$ of $K^n\times \ZZ^\ell\times \prod_{i=1}^N RV_{m_i}$ a definable set whenever
its natural pullback in $K^{n+\ell+N}$ (coordinatewise under $\ord$ and $rv_{m_i}$)  is a definable set.

\begin{defi}[Monomial function]
\label{defi:monomial}
Let $C' \subset K^{n'}$ be an open centred cell, and
$C=C'\times \{(0,\ldots,0)\} \subset K^n$ be the associated centred cell.
We say that a definable map $f: C \to \Z$ is a monomial function if there exists an integer $m$,
and a definable map $f_0 : (RV_m)^{n'} \to \Z$ such that
$f = f_0 \circ \pi$ where $\pi : C \to (RV_m)^{n'}$ is  defined by
\[ (x_1,\ldots,x_n) \mapsto ( rv_m(x_1), \ldots, rv_m(x_{n'})).\]
\end{defi}

A monomial function is thus just a definable function induced by a function purely on the $RV_m$ side for some $m$.

The following lemma illustrates how monomial functions are useful to build new centred cells.

\begin{lemma}
\label{lemma:newcell}
Let $C = rv_m^{-1}(G) \subset K^n$ be an open centred cell.
Let $C' \subset K^{n+1}$ be a set of the form
\[C' = \{(y,t) \in C \times K \st |\alpha(y)| \sq_1 |t| \sq_2|\beta(y)| \ \and \ t\in \lambda Q_{m',n'} \}\]
with the $\sq_i$ either $<$ or no condition for $i=1,2$, and let us assume that $\ord \alpha$ and $\ord \beta$ are monomial functions on $C$.
Then $C'$ is an open centred cell.
\end{lemma}
\begin{proof}
It is well known that a set like $C'$ is definable, since the condition $Q_{m',n'}$ is a definable set. The centred cell comes with an $m$, as an $RV_m$-pullback, and so do the monomial functions $\ord \alpha$ and $\ord \beta$ come with integers $m_1$, $m_2$, witnessing the definition of monomial function.
Increasing some of these $m,m_1,m_2$ and $m'$ if necessary, there is no harm in assuming that they are all equal. Now the lemma follows easily.
\end{proof}

The following result forms part of the induction scheme for the proofs of Theorems \ref{theo:1} and \ref{theo:retract} and will be proved
together with these theorems in Section \ref{sec:induction}.


\begin{prop}[$Mon_n$]
\label{lemma:monomial}
Let $X\subset K^n$ be a definable set.
For each $i=1,\ldots, m$, let $f_i : X \to \Z$ be a definable map.
Then there exists a decomposition 
$\displaystyle X=\coprod_{j=1}^\ell A_j$ in disjoint definable sets
such that
for each index $j \in \{1,\ldots, \ell\} $ there is a definable isometry
$\varphi_j : K^n \overset{\sim}{\to} K^n$ such that
$\varphi_j^{-1}(A)$ is a centred cell and $f \circ \varphi_j$ is a monomial function.
\end{prop}

\begin{rem}
In \cite{Hal10,Hal11}, it is conjectured that the trees of balls $T(X)$ of definable sets $X \subset \Zp^n$ are characterised by a simple
combinatorial condition.
The trees satisfying this combinatorial condition are called trees of level $d$.
One can easily check that if $X \subset \Zp^n$ is a centred cell, its associated tree $T(X)$ is a tree of level $d$ where $d=\dim(X)$.
Hence Proposition \ref{lemma:monomial} implies that up to cutting a definable set in finitely many pieces, the Conjecture 1.1 of \cite{Hal10} holds.
\end{rem}

\subsection{Retraction of centred cells}

\begin{lemma}
\label{lemma:RVtoZ}
Let $C \subset (K^\times)^n$ be an open centred cell.
Put $G:= \ord(C) \subset \Z^n$ and  $C':= \ord^{-1}(G) \subset (K^\times)^n$.
Suppose that the function 
$$
x\mapsto (\ac_m(x_1),\ldots, \ac_m(x_n))
$$
is constant on $C$.
Then there is a definable Lipschitz retraction $r$ from $C'$ to $C$.
\end{lemma}

\begin{proof}
The set $(\OK^\times / 1+\MK^m)$ is finite, of size
$N= (q-1)q^{m-1}$.
Let
\[\xi_1, \xi_2, \ldots , \xi_N \in \OK^\times \]
be a set of representatives of $(\OK^\times / 1+\MK^m)$ with the extra condition
\begin{equation}
\label{xi1}
\xi_1=1.
\end{equation}
For  $(u_1,\ldots, u_n) \in (RV_m)^n$ let us set
\begin{align*}
(\gamma_1, \ldots, \gamma_n) & := \ord(u_1,\ldots, u_n)\in \Z^n \\
A & := rv_m^{-1}(u_1\ldots, u_n) \subset (K^\times)^n\\
B & := \ord^{-1}(\gamma_1, \ldots, \gamma_n)\subset (K^\times)^n.
\end{align*}
One has
the following decomposition
\begin{equation}
B = \coprod_{(i_1 \ldots i_n) \in \{1\ldots N\}^n}  (\xi_{i_1}, \ldots, \xi_{i_n}) \cdot A .
\end{equation}
By definition of $C$ and $C'$, if $x\in C'$, there exists a unique
$n$-uple $(i_1 \ldots i_n)\in \{1,\ldots,N\}^n$ such that
$(\xi_{i_1}x_1, \ldots, \xi_{i_n}x_n)$ lies in $C$.
We define $r$ as follows.
\[
\begin{array}{cccc}
r: & C' & \to & C \\
   & x & \mapsto &    (\xi_{i_1}, \ldots, \xi_{i_n})\cdot  x
\end{array}
\]
where  $ (i_1 \ldots i_n)$ is the unique $n$-uple of $\{1,\ldots,N\}^n$
such that $(\xi_{i_1}, \ldots, \xi_{i_n}) \cdot x\in C$.

\par
Let us check that $r$ is Lipschitz by a disjunction case.
Let $x=(x_1, \ldots, x_n)$ and $y=(y_1,\ldots, y_n)$ be in $C'$.
Let $(\xi_{i_1}, \ldots, \xi_{i_n})$, resp.  $(\xi'_{i_1}, \ldots, \xi'_{i_n})$, be the
$n$-uple that appears in the definition of $r$ for $x$, resp.~for $y$.
Let us fix some index $j\in \{1, \ldots ,n\}$, and let us check that
$|r(x)_j - r(y)_j|  \leq |x_j - y_j|$.

\par
\textbf{Case 1:} $|x_j|=|y_j|$.

\par
\textbf{Case 1.1:}  $\acm(x_j)=\acm(y_j)$.

\par
In this case the constancy condition of the lemma implies that
$\xi_{i_j} = \xi'_{i_j}$.
So
\[|r(x)_j-r(y)_j| = | \xi_{i_j}(x_j-y_j)| = |x_j-y_j|.\]
\par
\textbf{Case 1.2:} $\acm(x_j) \neq \acm(y_j)$.

\par
The case condition implies that
\begin{equation}
\label{eq:pim}
|\varpi^{m}|\cdot |x_j| \leq  |x_j-y_j|
\end{equation}
By the constancy hypotheses, and according to the definition of $r$:
\[\acm(r(x)_j)  = \acm(r(y)_j)\]
and
\[ |r(x)_j| = |x_j| =|y_j|  =|r(y)_j|.\]
It follows that
\begin{equation}
\label{eq:rvmxy}
rv_m(r(x)_j) = rv_m(r(y)_j).
\end{equation}
So \eqref{eq:pim} and \eqref{eq:rvmxy} imply that
\[|r(x)_j- r(y)_j| \overset{\eqref{eq:rvmxy}}{\leq}  |\varpi^m|\cdot |x_j| \overset{\eqref{eq:pim}}{\leq}  |x_j-y_j|.\]

\par
\textbf{Case 2:}$|x_j| <|y_j|$.

\par
In this case we have
\[|r(x_j)|=|x_j| < |y_j|=|r(y)_j|.\]
So,
\[ |r(x)_j-r(y)_j| =|r(y)_j| = |y_j| = |x_j-y_j|,\]
and we are done.
\end{proof}

\begin{cor}
\label{cor:RVtoZ}
Let $C \subset (K^\times)^n$ be an open centred cell.
Put $G:= \ord(C) \subset \Z^n$ and $X:= \ord^{-1}(G) \subset (K^\times)^n$.
Then there exists a definable Lipschitz retraction $r: X \to C$.
\end{cor}

\begin{proof}
Since $(\OK^\times / 1+\MK^m)$ is finite, there exists a finite partition
$\displaystyle C = \coprod_{j\in J} C_j$ such that for each $j\in J$,
$C_j$ is an open centred cell and such that the map
 $(\overline{ac}_m)^n: C_j \to (\OK^\times / 1+\MK^m)$ is constant.
Thanks to Lemma \ref{lemma:RVtoZ}, for each $j\in J$ there exists a definable Lipschitz retraction
$r_j : X \to C_j$.
Finally, by Lemma \ref{lemma:glue}, there exists a definable Lipschitz retraction $r: X \to C$.
\end{proof}

\begin{lemma}
\label{lemma:clebis}
Let $U \subset \Z^{n-1}$ be a Presburger set.
Let $a,b $ be positive integers.
$\alpha,\beta : U\to b+a\Z$ be definable functions.
Let us set
\begin{align*}
V' & =\{ (u,v)\in U\times \ZZ \st \alpha(u) \sq_1 v \sq_2 \beta(u) \}, \\
V & = V' \cap \big(\Z^{n-1} \times (b + a\Z)\big),\\
\end{align*}
where $\sq_i$ is $<$ or no-condition for $i=1,2$.
Then there exists a definable Lipshitz retraction
$r: \ord^{-1}(V') \to \ord^{-1}(V)$.
\end{lemma}

\begin{proof}
Define $r$ by
\[
\begin{array}{cccc}
r : & \ord^{-1}(V') & \to & \ord^{-1}(V) \\
   & (x_1, \ldots, x_n) & \mapsto & (x_1,\ldots, x_{n-1},\varpi^i x_n)
\end{array}
\]
where $i$  is the unique index $i\in \{0, \ldots, a-1\}$ such that $(x_1,\ldots, x_{n-1},\varpi^ix_n) \in \ord^{-1}(V)$.
Thanks to the definitions of $V$ and $V'$, $r$ is well defined.
Let us prove that $r$ is Lipschitz. \par
So let $x=(x_1,\ldots,x_n)$ and $x'=(x'_1,\ldots, x'_n) \in \ord^{-1}(V')$, and let us prove that
$|r(x)-r(x')| \leq |x-x'|$.
Since $r$ does not change the first $n-1$ coordinates, it suffices to check that
$|r(x)_n - r(x')_n| \leq |x_n-x'_n|$.
Let $i$ (resp. $i'$) be the integer such that
$r(x)_n = \varpi^ix_n$ (resp. $r(x')_n = \varpi^{i'}x'_n$).\par
\textbf{Case 1:} $|x_n| =|x'_n|$.

\par
In this case $i=i'$ because $i$ (resp. $i'$) is the smallest integer $\geq 0$ such that $\ord(x_n)+i \in b+ a\Z$
(resp. $\ord(x'_n) +i' \in a\Z+b$).
So
\[| r(x)_n - r(x')_n| = |\varpi^i(x_n-x'_n)| \leq |x_n-x'_n|.\] \par
\textbf{Case 2:} $|x_n| >|x'_n|$.

\par
We have
\[|r(x)_n-r(x')_n)|  \leq \max(|\varpi^ix_n|, |\varpi^{i'}x'_n|)   \leq  |x_n| =  |x_n-x'_n|,\]
which finishes the proof.
\end{proof}

\subsection{Cell decomposition and preparation with Lipschitz centers}

In this section, we improve Proposition 4.6  of \cite{ClHa} (and Proposition 2.4 of \cite{ClCoLo}) by adding a kind of preparation to a cell decomposition statement. Instead of reproving 4.6 completely and observing that the preparation can be ensured as required, we give a blueprint on how to derive preparation from cell decomposition, in our context. This blueprint does not yet seem to work in more general P-minimal structures as in \cite{HaMaAver}. Recall that the idea of cell decomposition/preparation in the $p$-adic context goes back to P.~Cohen \cite{Cohen} and J.~Denef \cite{Denef}.
Let us first recall the notion of cells in our context, slightly adapted from the notion of \cite[Definition~3.1]{ClCoLo}.

\begin{defi}[$p$-adic cells]\label{def::cell}
Let $Y$ be a definable set of $\Z^M\times K^N$.
A cell $A\subset K\times Y$ over $Y$ is a (nonempty) set of
the form
 \begin{equation}\label{cel}
A=\{(t,y)\in K\times Y\mid y\in Y',\ \alpha(y) \sq_1 \ord (t-c(y)) \sq_2
\beta(y),\
  t-c(y)\in \lambda Q_{m,n}\},
\end{equation}
with $Y' \subset Y$ a definable set,
constants $n>0$, $m>0$, $\lambda$
in $K$, $\alpha,\beta\colon Y'\to \ZZ$ and $c\colon Y'\to K$ all
definable functions, and $\sq_i$ either $<$ or no
condition, and such that $A$ projects surjectively onto $Y'$.
 We call $c$ the center of the cell $A$, $\lambda Q_{m,n}$ the coset
of $A$, $\alpha$ and $\beta$ the boundaries of $A$, the $\sq_i$ the boundary conditions of $A$, and $Y'$ the
base of $A$. If $\lambda=0$ we call $A$ a $0$-cell
and if $\lambda\not=0$ we call $A$ a $1$-cell.
\end{defi}

\begin{theo}[$p$-adic cell decomposition/preparation with Lipschitz centers]
\label{theo:celldec}
Let $Y$ be a definable set of $\Z^M\times K^N$ and $X \subset Y\times K^{n}$ be a definable set, and $f_i:X \to \ZZ$ be definable functions, for $i=1\ldots m$.
Then there exists a partition $\displaystyle X=\coprod_{j=1}^\ell X_j$
such that for each $j$, for a well chosen coordinate projection $\pi_j: Y\times K^{n}\to Y\times K^{n-1}$, one has
\begin{itemize}
\item[a)] $X_j$ is a cell over $\pi_j(Y\times K^{n})$ with center $c_j: \pi_j(X_j)  \to K$ which is Lipschitz with respect to the variables of $K^{n-1}$
(i.e. for all $y\in Y$, the map $c_j(y,\cdot)$ is Lipschitz on $\pi_j(X_j)_y$)
and with coset $\lambda_j Q_{m_j,n_j}$ for some $\lambda_j\in K$ and positive integers $m_j$, $n_j$.
\item[b)] for all $i=1\ldots m$ and $j=1\ldots \ell$,
there exists a rational number $a_{i,j}\in \frac{1}{n_j}\Z$ and a definable map $h_{i,j}: \pi_j(X_j) \to \Z$ such that for all $(y,x) \in X_j$
\begin{equation}\label{prep}
 f_i(y,x) = h_{i,j}(y,\hat x) + {a_{i,j}}\ord\big(\frac{x_n-c_j(y,\hat x)}{\lambda_j}\big),
 \end{equation}
\end{itemize}
where we write $(y,x)=(y,\hat x,x_n)=(y,\pi_j(x),x_n)$ and with the convention that $0/0=1$ in (\ref{prep}).
\end{theo}

The argument we will give derives preparation from cell decomposition and some additional properties which seem to be unknown in more general P-minimal structures than our structures. We now give two of these additional properties.

 \par
A definable function $f: \ZZ^m\times K^n\to \ZZ^\ell$ is piecewise (with definable pieces) equal to the restriction to the piece of a function
$$
\gamma(z) + h(x)
$$
for definable functions $\gamma:\ZZ^m\to \ZZ^\ell$ and $h:K^n\to \ZZ^\ell$, by quantifier elimination.
\par

Similarly, a definable function $g: \ZZ^m\times K^n\to K^\ell$ is piecewise (with definable pieces) induced by definable functions $h:K^n\to K^\ell$, again by quantifier elimination.

These properties will be used in the following proof of Theorem \ref{theo:celldec}. 
We first give an extra definition (by induction on $n$) and a lemma.

\begin{defi}\label{def:skel}
Let $Y$ and $X \subset Y\times K^{n}$ be definable sets. We call $X$ a full cell over $Y$ with full centers $(c_1,\ldots, c_n)$ if it is a cell over $Y\times K^{n-1}$ with center $c_n$ and if the base of $X$ is itself a full cell over $Y$ with full centers $(c_1,\ldots, c_{n-1})$. For such a full cell $X$ over $Y$, the image in $Y\times (\ZZ\cup\{+\infty\})^n$ of $X$ under the map
$$
(y,z)\in X\mapsto (y, \ord (z_1-c_1(y)) ,\ldots, \ord (z_n-c_n(y,z_1,\ldots,z_{n-1}))
$$
is called the skeleton of the full cell $X$ over $Y$.
\end{defi}

\begin{lemma}\label{lem:skel}
Let $X\subset \ZZ^m\times K^n$ be a full cell over $\ZZ^m$ with skeleton $A\subset \ZZ^{m+n}$ over $\ZZ^m$. Suppose that for each natural number $N$, there are infinitely many tuples $z\in \ZZ^m$ such that $A_z:=\{w\in \ZZ^n\mid (z,w)\in A\}$ is of cardinality at least $N$. Then there are $z$ and $z'$ in $\ZZ^m$ with $z\not=z'$ such that $X_{z}:=\{x\in K^n\mid ( z,x)\in X\}$  and $X_{z'}$ have nonempty intersection.
\end{lemma}
\begin{proof}
The lemma follows from the cell decomposition theorem for Presburger sets, namely Theorem 1 of \cite{CPres}, and the property about $g$ mentioned just before Definition \ref{def:skel}, applied to the occurring centers of the full cell $X$.
\end{proof}

\begin{proof}[Proof of Theorem \ref{theo:celldec}]
First note that the result including a) but without b) is Proposition 4.6 of \cite{ClHa}, for a general definable subset $Y$ of $\Z^M\times K^N$ for $M,N\geq 0$.
\par

    Clearly we are allowed to work piecewise, and hence, by induction on $m$ we may assume that the $f_i$ do not vanish on $X$.
\par

The simple case that the function $x\in X_y \mapsto f_i(y,x)$ is constant for each $y\in Y$ and each $i$ is immediate, where $X_y = \{x\in K^n\mid (y,x)\in X\}$.
\par

Let us now consider the graph $X'$ of the function
$$
F:(y,x)\in X\mapsto (\ord f_1(y,x),\ldots,\ord f_m(y,x))\in \ZZ^m
$$
and let us put $Y' := Y\times \Z^m$ so that $X'$ can be naturally seen as a subset of $Y'\times K^n$. Let $f_i'$ be the induced function on $X'$ coming from $f_i$ and the natural bijection between $X$ and $X'$. By the simple case treated above, the theorem holds for $X'$ and the functions $f_i'$, yielding cells $X_j'$, functions $h_{i,j}'$ and so on. We now derive from this the result for $X$ and the $f_i$. By the induction hypothesis (with induction on $n$) applied recursively to the base of the cell $X'$, we may suppose that $X'$ is a full cell over $Y'$. Let $A'$ be its skeleton. By the graph construction, the fibers $X'_{z,y}:=\{x\in K^n\mid (z,y,x)\in X'\}$ are all disjoint when $z$ runs over $\ZZ^m$ and when $y$ is any fixed value in $Y$. By Lemma \ref{lem:skel}, this implies that we can partition $X'$ further without changing the centers and reduce to the case that $A'_{z,y}:=\{s\in \ZZ^n\mid (z,y,s)\in A'\}$ is either empty or a singleton for each $z\in \ZZ^m$ and each $y\in Y$. Now the theorem for $X$ and the $f_i$ follows from the piecewise linearity of Presburger definable functions of the cell decomposition theorem for Presburger sets, namely Theorem 1 of \cite{CPres}. Indeed, the dependence of $z$ on $s$ for any fixed $y$ under the condition $(z,y,s)\in A'$ is a piecewise linear function, uniformly so in $y$.
\end{proof}

\section{Proofs of the main result by a joint induction}\label{sec:proof}

We prove Theorems \ref{theo:1} and \ref{theo:retract}, and Proposition \ref{lemma:monomial} by a joint induction, that is, we prove the properties 
$(Mon_n)$,  $(Ret_n)$ and $(Ext_n)$ by induction on $n$.

\subsection{Proofs for the case $n=1$}
\label{section:basis}

When $n=1$, Proposition \ref{lemma:monomial} follows from the $p$-adic cell decomposition (Theorem \ref{theo:celldec}).
For this, one does not need the Lipschitz assertion on the centers of the cells. \par
To prove Theorem \ref{theo:retract} for $n=1$, thanks to
Lemma \ref{lemma:glue}, Theorem \ref{theo:celldec}, and up to translating with the center, one is reduced to find a definable Lipschitz retraction
$r: K \to \overline{C}$ where $C \subset K^\times$ is an open centred cell.
By Corollary \ref{cor:RVtoZ}, one can assume that $C= \ord^{-1}(G)$ for a Presburger set $G \subset \Z$.
Such a set is a finite union of sets of the form
\[\{ g\in \Z \st g\in a\Z +b, \  \and \ \alpha \sq_1 g  \sq_2 \beta \}\]
where $a,b,\alpha,\beta \in \Z$ and each $\sq_i$ is $<$ or no condition.
Thanks to Lemma \ref{lemma:clebis}, we can drop the congruence relation $g \in a\Z +b$ and assume that
\[G = \{g\in \Z \st  \alpha \sq_1 g  \sq_2 \beta \}\]
and are reduced to construct a definable Lipschitz retraction $r: K \to \overline{C}$.
Depending on the values of the $\sq_i$'s (namely $<$ or no-condition) this leaves four cases:
\begin{enumerate}
\item
$C = K^\times$. Then, $\overline{C}=K$, and we take $r=id$.
\item
$C=\{x\in K \st 0 <|x| \leq s\}$, so $\overline{C}=\{x\in K \st |x|\leq s\}$.
Then we take
\[
\begin{array}{cccc}
r: & K & \to & \overline{C} \\
& x &\mapsto &
\begin{cases}
x & \text{if} \ |x| \leq s \\
0 & \text{if} \ |x| >s
\end{cases}
\end{array}
\]
\item
$C=\{x\in K \st  s \leq |x| \leq s'\}=\overline{C}$.
Let us pick $x_0 \in K$ such that $|x_0|=s$.
Then we consider
\[
\begin{array}{cccc}
r: & K & \to & \overline{C} \\
& x &\mapsto &
\begin{cases}
x & \text{if} \ |x| \in C \\
x_0 & \text{otherwise}
\end{cases}
\end{array}
\]
\item
$C=\{x\in K \st  s \leq |x| \}=\overline{C}$.
Let us pick $x_0 \in K$ such that $|x_0|=s$.
Then we consider
\[
\begin{array}{cccc}
r: & K & \to & \overline{C} \\
& x &\mapsto &
\begin{cases}
x & \text{if} \ |x| \in C \\
x_0 & \text{otherwise}
\end{cases}
\end{array}
\]
\end{enumerate}
One easily checks the required conditions in each of these cases.
This proves Theorem \ref{theo:retract} when $n=1$,
and implies Theorem \ref{theo:1}  for $n=1$ as explained on page \pageref{proof:theo1}.

\subsection{Proofs for $n>1$}\label{sec:induction}
We now prove by induction the properties 
$(Mon_n)$,  $(Ret_n)$ and $(Ext_n)$ listed in the introduction.
The basis of the induction has been obtained in Section \ref{section:basis}.\par
So let us fix an integer $n>1$ and let us assume that $(Ret_{n-1})$,  
$(Mon_{n-1})$ and  $(Ext_{n-1})$ hold.
We will prove the statements in the following order:  
$(Mon_n)$,  $(Ret_n)$ and $(Ext_n)$.

\begin{proof}[Proof of Proposition \ref{lemma:monomial} $(Mon_n)$.]
$ $\par
Let us apply Theorem \ref{theo:celldec} to the functions $f_i$.
So we can assume that $X \subset K^n$ is a cell
\[X=\{ (y,t)\in Y\times K \st |\alpha(y)| \sq_1|t-c(y)| \sq_2 |\beta(y)| \ \and \ t-c(y)\in \lambda Q_{m',n'} \} \]
over $K^{n-1}$ with a base $Y$ and a Lipschitz center $c:Y\to K$, and that for each $i \in \{1,\ldots,m\}$ there is
some $a_i\in \Q$ and $h_i: Y \to K$ some definable function such that,
for each $y\in Y, t\in K$ with $(y,t)\in X$, one has that $f_i(y,t)$ is of a prepared form, coming from (\ref{prep}).
By induction hypothesis $(Ext_{n-1})$, we can extend $c:Y \to K$ to a definable Lipschitz
function $\overline{c}: K^{n-1} \to K$.
Let us consider the definable isometry
$\varphi: K^n \to K^n$ defined by $ (y,t) \in K^{n-1}\times K \mapsto (y,t+\overline{c}(y))$.
Then, considering $\varphi^{-1}(X)$, we can assume that $c\equiv 0$.
Hence, for any $(y,t) \in X$ one has
\[ f_i(y,t) = h_i(y) + a_i \ord (t/\lambda),\]
for some rational numbers $a_i$ and some definable functions $h_i$.
By induction hypothesis $(Mon_{n-1})$ applied to the functions $h_i, \ord \alpha, \ord \beta$, and up to cutting $X$ in finitely many pieces, we can assume that
\begin{itemize}
\item
$Y\subset K^{n-1}$ is a centred cell.
\item $X$ is the cell over $Y$ defined by:
\[X=\{ (y,t)\in Y\times K \st |\alpha(y)| \sq_1|t| \sq_2 |\beta(y)| \ \and \ t\in \lambda Q_{m',n'} \} \]
\item for each index $i$ and each $(y,t)\in X$ one has:
\[ f_i(y,t) = h_i(y) + a_i \ord (t/\lambda).\]
\item The functions $h_i$, $\ord \alpha$, and $\ord \beta$ are monomial functions on $Y$.
\end{itemize}
Then according to Lemma \ref{lemma:newcell}, $X$ is a centred cell, and each $f_i$ is a monomial function.
\end{proof}

\begin{proof}[Proof of Theorem \ref{theo:retract} $(Ret_n)$.]
$ $ \par
\textbf{Step 1.} Let us show first that when $\dim(X) <n$, there exists a definable Lipschitz retraction $r : K^n \to \overline{X}$.

\par
By Lemma \ref{lemma:glue}, we can cut $X$ in definable pieces.
Thanks to Proposition \ref{lemma:monomial}, we can assume that
there exists a definable set $X' \subset K^{n-1}$ such that $X = X' \times \{0\}$.
So we can apply our induction hypothesis ($Ret_{n-1}$) to $X'$ and conclude.\par
\textbf{Step 2.} According to $(Mon_n)$, we can assume that $X=\overline{C}$ where $C$ is a centred cell, let us say of the form
$C=C' \times \{(0,\ldots,0)\}$ for some open centred cell $C' \subset K^{n'}$. \par
If $n'<n$, we are reduced to Step $1$. \par
If $n=n'$, then $C$ is an open centred cell.
So according to Corollary \ref{cor:RVtoZ}, we can assume that
\[X = \ord^{-1}(G)\]
for some Presburger set $G\subset \Z^n$.\par
\textbf{Step 3.}
For each $i \in \{1,\ldots,n\}$, let
\[G_i:= \{(g_1,\ldots,g_n)\in G \st g_i \geq g_j \ \text{for all} \ j=1, \ldots, n \}.\]
Then
\[ G = \bigcup_{i=1}^n G_i \]
and
\[X = \bigcup_{i=1}^n \ord^{-1}(G_i). \]
By Lemma \ref{lemma:glue}, we are reduced to prove the result for each $X= \ord^{-1}(G_i)$.
So, replacing $G$ by the $G_i$'s, and up to a permutation of the coordinates, we can assume that $G$ satisfies the following condition
\begin{equation}
\label{condition1}
g_n \geq g_j \mbox{ for all $j=1, \ldots, n$ and all   $(g_1,\ldots,g_n) \in G$}.
\end{equation}
Using the results mentioned in Section \ref{section:Presburger}, one can assume that $G$ still satisfies condition \eqref{condition1} and is moreover of the form
\[ G =\{ (u,v)\in \Z^{n} \st u\in U,\ v\in a\Z+b \ \and \ \alpha(u) \sq_1 v \sq_2 \beta(u) \}\]
where $U \subset \Z^{n-1}$ is a Presburger set, $\alpha,\beta: U \to \Z$ are definable functions, and each $\sq_i$ is $<$ or no condition.
Then, using Lemma \ref{lemma:clebis}, we can remove the congruence condition $v\in a\Z+b$.
\par
\textbf{Step 4.}
We are reduced to the following assertion.
Given the following data
\begin{itemize}
\item a Presburger set $U \subset \Z^{n-1}$;
\item  definable functions
$\alpha,\beta : U \to \Z$ such that for any $u\in U$ there exists some $v\in \Z$ with $\alpha(u) < v < \beta(u)$ and such that
for any $(u_1, \ldots, u_{n-1}) \in U$, we have, by (\ref{condition1}),
\[\alpha(u) \geq  \max_{i=1,\ldots, n-1} u_i;\]
\item a Presburger set
\[ G =\{ (u,v)\in \Z^{n} \st u\in U, \ v\in \Z \ \and \ \alpha(u) \sq_1 v \sq_2 \beta(u) \}.\]
\end{itemize}
We then have to show that there exists a definable Lipschitz retraction
\[r : K^{n} \to \overline{\ord^{-1}(G)}.\]
\par
\textbf{Step 4.1}. We first construct a definable Lipschitz map
\begin{equation}\label{r:partial}
r:\ord^{-1}( U \times \Z) \to \overline{\ord^{-1} (G)}
\end{equation}
such that $r(x)=x$ for all $x\in \ord^{-1}(G)$.
\par
If both $\sq_1$ and $\sq_2$ are no condition, then $G= U\times \Z$ and the  map
\[r : x\in \ord^{-1}(U\times \Z) \mapsto x\]
works. So we can assume that one of $\sq_1$ or $\sq_2$ is $<$.

\par
Recall that the theory of our one-sorted structure $K$ has definable Skolem functions, see \cite{VDDSko}[Theorem 3.2] in the semi-algebraic case and \cite{DenVDD} in the subanalytic case, but this also follows directly from the above cell decomposition result in any of our settings.
If $\sq_2$ is $<$, we let
\[f_- : \ord^{-1}(U) \to K^\times\]
be a definable Skolem function satisfying
\[\ord(f_-(u)) =  \beta(u) - 1 \ \text{for all} \  u \in \ord^{-1}(U).\]
Similarly, if $\sq_1$ is $<$, we let
\[f_+ : \ord^{-1}(U) \to K^\times\]
be a definable Skolem function which satisfies
\[\ord(f_+(u)) = \alpha(u) + 1 \ \text{for all} \ u\in \ord^{-1}(U).\]
Let us set write $H$ for the union of the graph of $f_-$ (when $f_-$ is defined) and the graph of $f_+$ (when $f_+$ is defined), namely,
\begin{align*}
H & =\{ (u,f_+(u)) \st u\in \ord^{-1}(U) \} \cup \{ (u,f_-(u)) \st u\in \ord^{-1}(U) \} \ \ \text{when} \sq_1 \text{and} \sq_2 \text{are} < \\
H & =\{ (u,f_+(u)) \st u\in \ord^{-1}(U) \} \ \   \text{when} \sq_1 \text{is} < \text{and} \sq_2 \text{is no condition} \\
H & =\{ (u,f_-(u)) \st u\in \ord^{-1}(U) \}    \ \ \text{when} \sq_1 \text{is no condition and} \sq_2 \text{is} <.
\end{align*}
By construction, $H \subset \ord^{-1}(G)$, so $\overline{H} \subset \overline{\ord^{-1}(G)}$ and
$\dim(H) =\dim(\overline{H})=n-1$.
According to step 1, we can find
a definable  Lipschitz retraction
\[s : K^{n} \to \overline{H}.\]

\par
We now define our Lipschitz map as desired for (\ref{r:partial}) like this:
\[ \begin{array}{cccc}
r : & \ord^{-1}(U \times \Z) & \to & \overline{\ord^{-1} (G)} \\
&   z & \mapsto &
\begin{cases}
z &  \text{if} \ z \in  \ord^{-1} (G) \\
s(z) & \text{if} \ z \notin  \ord^{-1} (G)
\end{cases}
\end{array}
\]
The key remaining work is to prove that $r$ is Lipschitz.
Let us consider $z=(y,x)$ and $z'=(y',x') \in \ord^{-1}(U \times \Z)$ and let us prove that
$|r(z) - r(z')|\leq |z-z'|$.\par
If $z=(y,x)$ and $z'=(y',x')$ belong simultaneously to $\ord^{-1}(G)$, or to its complement $\ord^{-1}(G)^c$, then
$|r(z)-r(z')| \leq |z-z'|$ because the identity map and the function $s$ are Lipschitz.
So we will assume that
$\ord(z) \in G$ and $\ord(z') \notin G$.
\par
Since $|(y,x) -(y',x')| = \max (|y-y'| , |x-x'|)$ it follows that
$|(y,x)-(y',x)|$ and $|(y',x)-(y',x')|$ are less or equal than
$|(y,x) -(y',x')|$.
So we can assume that $x=x'$ or $y=y'$.
\par
\textbf{Case 1}: $x=x'$.
So $z=(y,x) \in \ord^{-1}(G)$ and
$z'=(y',x) \notin \ord^{-1}(G)$.
For simplicity of notation, let us assume that
$|y_1| \neq |y'_1|$.
In particular, $|z-z'| \geq |y_1|$.
If $\sq_2$ is $<$, let $z''=(y,f_-(y)) \in \ord^{-1}(G)$
(otherwise $\sq_2$ is no condition, $\sq_1$ is $<$, then we set $z''=(y,f_+(y)) \in \ord^{-1}(G)$ and the undermentioned reasoning can also been applied).
Since $z,z'' \in \ord^{-1}(G)$, according to condition \eqref{condition1}, $|f_-(y)| \leq |y_1|$
and $|x| \leq |y_1|$.
So
\[|z-z''| = |x-f_-(y)| \leq |y_1| \leq |z-z'|.\]
Likewise
\[|z'' -z'| = |(y,f_-(y)) - (y',x)| = \max (|y-y'| , |x-f_-(y)|) \]
\[
\leq \max (|y-y'|, |y_1|) \leq |z-z'|\]
So it suffices to show that
$|r(z) -r(z'')| \leq |z-z''|$ and that
$|r(z'') -r(z')| \leq |z''-z'|$.
Since $z,z'' \in \ord^{-1}(G)$,
$r(z)=z$ and $r(z'')=z''$ and this implies the first inequality.
Finally, $z'' \in H$ and $z'\notin \ord^{-1}(G)$ so by definition of
$r$,
$r(z'') = s(z'')$ and $r(z') =s(z')$.
Since $s$ is Lipschitz,
$|r(z'')-r(z')| = |s(z'')-s(z')| \leq |z''-z'|$.
\par
\textbf{Case 2} : $y=y'$.
So $z=(y,x) \in \ord^{-1}(G)$ and $z'=(y,x')\notin \ord^{-1}(G)$.
\par
\textbf{Case 2.1}. $\sq_2$ is the condition $<$ and $\beta (\ord y) \leq \ord x'$,
where $\ord y$ stands for $(\ord y_1,\ldots, \ord y_{n-1})$.

\par
So $|z-z'| =|x-x'|=|x|$.
Let us consider $z''=(y,f_-(y))$.
Then
\[|x'| < |f_-(y)| \leq |x|.\]
So
\[|z'-z''| = |f_-(y)| \leq |x| \leq |z-z'|.\]
Likewise,
\[|z''-z| =|f_-(y)-x| \leq  |x| \leq |z-z'|.\]
So it suffices to show that $|r(z'') -r(z')| \leq |z''-z'|$
and that $|r(z) -r(z'')| \leq |z-z''|$.
The first inequality is true because $r(z')=z'$ and $r(z'')=z''$ since $z',z'' \in \ord^{-1}(G)$.
The second inequality is true because
$z'' \in H$ and $\ord(z) \notin G$ so by definition of $r$,
$r(z'')=s(z'')$ and $r(z)=s(z)$ and $s$ is Lipschitz. \par
\textbf{Case 2.2}. $\sq_1$ is the condition $<$ and $\alpha (\ord y) \geq \ord x'$,
where $\ord y$ stands for $(\ord y_1,\ldots, \ord y_{n-1})$.

\par
So $|z-z'| =|x-x'|=|x'|$.
Let us consider $z''=(y,f_+(y))$.
Then
\[|f_+(y)|  < |x'|.\]
So
\[|z'-z''| = |x'-f_+(y)| =|x'| = |z-z'|.\]
Likewise,
\[|z''-z| =|f_+(y)-x| \leq |f_+(y)|  <|x'| = |z-z'|.\]
So it suffices to show that $|r(z'') -r(z')| \leq |z''-z'|$
and that $|r(z) -r(z'')| \leq |z-z''|$.
This is shown as at the end of Case 2.1.

\textbf{Step 4.2}.
By induction hypothesis ($Ret_{n-1}$), there is a definable Lipshitz retraction
\[\sigma : K^{n-1} \to \overline{\ord^{-1}(U)}.\]
It induces a definable Lipschitz retraction
\[
\begin{array}{cccc}
\tau : &  K^{n} & \to &\overline{\ord^{-1}( U \times \Z)} = \overline{\ord^{-1}(U)} \times K  \\
 & (x_1,\ldots,x_n) & \mapsto & (\sigma(x_1,\ldots,x_{n-1}),x_n)
 \end{array}
\]
Using lemma \ref{lemma:extension}, we can extend $r$ of (\ref{r:partial}) as constructed in Step 4.1 in a Lipschitz way to
\[\overline{r} :\overline{\ord^{-1}(U\times \Z)} \to \overline{\ord^{-1}(G)}\]
which is a definable Lipschitz retraction.
Then $\overline{r} \circ \tau : K^{n} \to \overline{\ord^{-1}(V)}$ is a definable Lipschitz retraction as desired for property $(Ret_n)$.
\end{proof}

\begin{proof}[Proof of Theorem \ref{theo:1} $(Ext_n)$.]
Property $(Ext_n)$ follows from Property $(Ret_n)$ similarly as how Theorem \ref{theo:1} is obtained from Theorem \ref{theo:retract} just below Theorem \ref{theo:retract}.
\end{proof}

To end the paper, we give family versions of Theorems \ref{theo:1} and \ref{theo:retract}.

\begin{theo}
\label{theo:1:fam}
Let $Y\subset K^m$ and $X \subset Y\times K^n$ be definable sets.
Let $f:X \to K^\ell$ be a definable and suppose that for each $y\in Y$, the function $f_y$ sending $z$ to $f(y,z)$ is a Lipschitz function on $X_y:=\{z\mid (y,z)\in X\}$.
There exists a definable function $\widetilde f: Y\times K^n \to K^\ell$ such that, for each $y\in Y$,
$$
\widetilde f_y:z\mapsto \widetilde f(y,z)
$$
is a Lipschitz extension of $f_y$.
\end{theo}

\begin{theo}
\label{theo:retract:fam}
Let $Y\subset K^m$ and $X \subset Y\times K^n$ be definable sets.
There exists a definable function $r:Y\times K^n\to Y\times K^n$ such that, for each $y\in Y$, $r_y$ is a retraction  from  $K^n$ to $\overline{X_y}$, with $X_y=\{z\mid (y,z)\in X\}$ and $\overline{X_y}$ its topological closure.
\end{theo}

These two theorems follow by noting that the proofs of Theorems \ref{theo:1} and \ref{theo:retract} and of Proposition \ref{lemma:monomial} work completely uniformly in $y\in Y$, and that definable Skolem functions can be used to pick the elements like $x_0$ in Section \ref{section:basis} for the case $n=1$. In these last theorems, $Y$ is not allowed to be a definable subset of $\ZZ^m\times K^{m'}$ since that would render impossible the above usage of definable Skolem functions, which indeed only exist in the one-sorted structure $K$.

\bibliographystyle{alpha}
\bibliography{bibli}
\end{document}